\documentclass{amsart}
\usepackage{amscd,amsfonts,amssymb,amsmath}
\usepackage[margin=3.9cm]{geometry}

\usepackage{amsmath}

\newtheorem{theorem}{Theorem}[section]
\newtheorem{cor}[theorem]{Corollary}
\newtheorem{lemma}[theorem]{Lemma}
\newtheorem{prop}[theorem]{Proposition}

\theoremstyle{definition}

\theoremstyle{remark}

\numberwithin{equation}{subsection}
\theoremstyle{plain}

\newtheorem{problem}{Problem}

\newcommand{\secref}[1]{Section~\ref{#1}}
\newcommand{\thmref}[1]{Theorem~\ref{#1}}
\newcommand{\lemref}[1]{Lemma~\ref{#1}}

\newcommand{\propref}[1]{Proposition~\ref{#1}}
\newcommand{\corref}[1]{Corollary~\ref{#1}}

\newcommand{\eqnref}[1]{~{\textrm(\ref{#1})}}
\numberwithin{equation}{section}

\begin{document}

\title[PALINDROMIC WIDTH OF FREE NILPOTENT GROUPS]{PALINDROMIC WIDTH OF FREE NILPOTENT GROUPS}

\author{Valeriy ~G.~Bardakov}
\address{Sobolev Institute of Mathematics, Novosibirsk, 630090, Russia}
\email{bardakov@math.nsc.ru}
\author{Krishnendu Gongopadhyay}
\address{Department of Mathematical Sciences, Indian Institute of Science Education and Research (IISER) Mohali,
Knowledge City, Sector 81, S.A.S. Nagar, P.O. Manauli 140306, India}
\email{krishnendu@iisermohali.ac.in, krishnendug@gmail.com}
\subjclass[2000]{Primary 20F18; Secondary 20D15,20E05}
\keywords{palindromic width, free nilpotent groups}

\thanks{\bf The authors gratefully acknowledge the support of the Indo-Russian DST-RFBR project grant DST/INT/RFBR/P-137}

\date{December 2, 2013}


\begin{abstract}
In this paper we consider the palindromic width of free nilpotent groups. In
particular, we prove that the palindromic width of a finitely generated free
nilpotent group is finite. We also prove that the palindromic width of a free abelian-by-nilpotent group is finite. 
\end{abstract}
\maketitle

\section{Introduction}
Let $G$ be a group with a set of generators $A$. A reduced word in the alphabet $A^{\pm 1}$
is a \emph{palindrome} if it reads the same forwards and backwards. The palindromic length $l_{\mathcal P}(g)$
of an element $g$ in $G$ is the minimum number $k$ such that $g$ can be expressed as a product of $k$
palindromes. The \emph{palindromic width} of $G$ with respect to $A$ is defined to be ${\rm pw}(G) = \underset{g \in G}{\sup} \ l_{\mathcal{P}}(g)$.
In analogy with commutator width of groups (for example see \cite{MR, AR, AR1, B}),
it is a problem of potential interest to study palindromic width of groups.
Palindromes of free groups have already proved useful in studying various aspects
of combinatorial group theory, for example see \cite{col, gj, hel}.
In \cite{BST}, it was proved that the palindromic width of  a non-abelian free
group is infinite.  This result was generalized in \cite{BT} where
the authors proved that almost all free products have infinite palindromic
width; the only exception is given by the free product of two cyclic groups of
order two, when the palindromic width is two. Piggot \cite{P} studied the relationship between primitive words and palindromes in free groups of rank two. It follows from \cite{BST, P} that up to conjugacy, a primitive word can
always be written as either a palindrome or a product of two palindromes and that certain pairs of
palindromes will generate the group. Recently Gilman and Keen \cite{gk1, gk2} have used tools from hyperbolic geometry to reprove this
result and further have obtained discreteness criteria for two generator subgroups in ${\rm PSL}(2, \mathbb C)$ using the geometry of
palindromes. The work of Gilman and Keen indicates a deep connection between palindromic width of groups and geometry.

Let ${\rm N}_{n, r} $ be the free nilpotent group of rank $n$ and of step $r$. In this paper we consider the palindromic width of free nilpotent groups. We prove that the palindromic width of a finitely generated free
nilpotent group is finite. In fact, we prove that the palindromic width of an arbitrary rank $n$ free nilpotent group is bounded by $3n$. For the $2$-step free nilpotent groups, we improve this bound. For the groups, ${\rm N}_{n, 1}$ and ${\rm N}_{2,2}$
we get the exact values of the palindromic width.   Our main theorem is the following.

\begin{theorem}\label{mainth}
Let ${\rm N}_{n, r} $ be the free $r$-step nilpotent group of rank $n \geq 2$. Then the  following holds:
\begin{enumerate}
\item{ The palindromic width ${\rm pw}({\rm N}_{n, 1})$ of a free abelian group of rank $n$  is equal to $n$. }
\item{For $r \geq 2$,
 $n \leq {\rm pw}({\rm N}_{n, r} ) \leq 3n$. }
\item{${\rm pw}({\rm N}_{n, 2}) \leq 3 (n-1)$. }
\end{enumerate}
\end{theorem}
We prove the theorem in \secref{fn1}. Along the way, we also prove ${\rm pw}({\rm N}_{2,2})=3$. In \secref{prel}, after recalling some basic notions and related basic results, we prove \lemref{l:32} which is a key ingredient in the proof of \thmref{mainth}. As a consequence of \lemref{l:32} we also prove that the palindromic width of a free abelian-by-nilpotent group of rank $n$ is bounded by $5n$, see \propref{fabn}. 
 For the group ${\rm N}_{3,2}$ it is possible to improve the bound given in (2) of the above theorem. In fact, 
\hbox{$4 \leq {\rm pw}({\rm N}_{3, 2}) \leq 6$}. A detailed proof of this fact will appear elsewhere. It would be interesting to obtain solutions to the following problems. 
\begin{problem}\label{prob1}
\begin{enumerate}
 \item{For $n \geq 3$, $r \geq 2$, find ${\rm pw}({\rm N}_{n,r} )$.}
\item{Construct an algorithm that determines $l_{\mathcal P} (g)$ for arbitrary $g \in {\rm N}_{n, r} $.}
\end{enumerate}
\end{problem}
The above problem can be asked for any other groups as well. In general, the palindromic width of an
arbitrary group depends on the generating set of the group. However, the advantage of working with free
nilpotent groups is that, in this case we have a basis and hence, the palindromic
width is the same regardless of the choice of a basis as a generating set. 

\medskip For $g$, $h$ in $G$, the \emph{commutator} of $g$ and $h$ is defined as $[g,
h]=g^{-1} h^{-1}gh$. If $\mathcal{C}$ is the set of commutators in some group
$G$ then the commutator  subgroup $G'$ is generated by $\mathcal{C}$.
The length $l_{\mathcal{C}}(g)$ of an element $g \in G'$ is called the
\emph{commutator length}.  The width
${\rm wid}(G', \mathcal{C})$ is called the \emph{commutator width} of $G$ and is
denoted by ${\rm cw}(G)$. It is well known \cite{R} that the commutator width of a
free non-abelian group is infinite, but the commutator width of a finitely
generated nilpotent group is finite (see \cite{AR,AR1}). An algorithm of the
computation of the commutator length in free non-abelian groups can be found in
\cite{B}. The following problem is natural to ask. 

\begin{problem}
Is it true that for a finitely generated group $G = \langle A \rangle$, the
palindromic width ${\rm pw}(G)$ is finite if and only if the commutator width ${\rm cw}(G)$
is finite?
\end{problem}

\section{Background and Preliminary Results} \label{prel}
 \subsection{Background} \ \
\subsubsection{Widths of groups.} Let $G$ be a group and $A \subseteq G$ a subset
that generates $G$. For each $g \in G$ define the {\it length} $l_A(g)$ of $g$
with respect to $A$ to be the minimal $k$ such that $g$ is a product of $k$
elements of $A^{\pm 1}$. The supremum of the values $l_A(g),$ $g \in G$, is
called the \emph{ width} of $G$ with respect to $A$ and is denoted by ${\rm wid}(G,
A)$. In particular,
${\rm wid}(G, A)$ is either a natural number or $\infty$. If ${\rm wid}(G, A)$
is a natural number, then every element of $G$ is a product of at
most ${\rm wid}(G, A)$ elements of $A$.

Let $A$ be a set of generators of a group $G$.  A reduced word $w$ in the
alphabet $A^{\pm 1}$ is called a {\it palindrome} if $w$ reads the same
left-to-right and right-to-left. An element $g$ of $G$ is called a
\emph{palindrome} if $g$ can be represented by some word $w$ that is a
palindrome in the alphabet $A^{\pm 1}$. We denote the set of all palindromes in
$G$ by $\mathcal{P} = \mathcal{P}(G)$. Evidently,  the set $\mathcal{P}$
generates $G$. Then any element $g \in G$ is a product of palindromes
$$
g = p_1 p_2 \ldots p_k.
$$
The minimal $k$ with this property is called the \emph{palindromic length} of $g$ and
is denoted by $l_{\mathcal{P}}(g)$.  The \emph{palindromic width} of $G$ is then given by
$$
{\rm pw}(G) = {\rm wid}(G, \mathcal{P}) = \underset{g \in G}{\sup} \ l_{\mathcal{P}}(g).
$$

\subsubsection{Free Nilpotent Groups} Let ${\rm N}_{n,r}$ be the free $r$-step nilpotent group of rank $n$ with a basis $x_1, \ldots,
x_n$.
 For example, when $r=1$,
${\rm N}_{n,1}$ is simply the free abelian group generated by $x_1, \ldots, x_n$, so
every element of ${\rm N}_{n,1}$ can be presented uniquely as
 $$g=x_1^{\alpha_1} \ldots x_n^{\alpha_n}$$
 for some integers $\alpha_1, \ldots, \alpha_n$. For $r=2$,  every element $g \in {\rm N}_{n,2}$ has the
form

 \begin{equation}\label{fr1}g=\prod_{i=1}^n x_i^{\alpha_i} \cdot \prod_{1 \leq j < i
\leq n} [x_i, x_j]^{\beta_{ij}} \end{equation}
 for some integers $\alpha_i$  and $\beta_{ij}$, where $[x_i, x_j] = x_i^{-1} x_j^{-1} x_i x_j$
 are basic commutators (see  \cite[Chapter 5]{MKS}).

\medskip For the free nilpotent group ${\rm N}_{n, r} $,  let ${\rm N}_{n, r} '$ be its commutator subgroup. We
note the following lemmas that will be used later. 
\begin{lemma} \label{ar1} \cite{AR1, MR}
Any element  $g$ in the commutator subgroup ${\rm N}_{n,r}'$ can be
represented in the form
$$
g = [u_1, x_1] \, [u_2, x_2] \, \ldots \, [u_n, x_n],~~~u_i \in {\rm N}_{n,r}.
$$
\end{lemma}

In fact,  Allambergenov and Roman'kov
\cite{AR1} proved the following.

\begin{itemize}
 \item[(i)]{ Any element of the commutator subgroup ${\rm N}_{n,2}'$ is a product of
no more than $[n/2]$ commutators.}

\item[(ii)]{ Any element of the commutator subgroup ${\rm N}_{n,r}'$ in all other
cases ($r \geq 3, \  n \geq 4$ or $r > 3, \ n =2$) is a product of no more than
$n$ commutators.}
\end{itemize}
\begin{lemma}\label{reh} \cite[Lemma 3]{MR}
Let $A$ be a normal subgroup of ${\rm N}_{n, r} $. If $A$ is abelian or $A$ lies in the second center of ${\rm N}_{n, r} $, then every element of $[A, {\rm N}_{n, r} ]$ has the form $[u_1, x_1] \, [u_2, x_2] \, \ldots \, [u_n, x_n]$ for some $u_i \in A$. 
\end{lemma}
\subsection{Preliminary Results}
Let $G = \langle A \rangle$ be a group and $\mathcal{P} = \mathcal{P}(A)$ be the
set of palindromes in $G$. Evidently,  any palindrome $p \in \mathcal{P}$
can be represented in the form
$$
p = u a^{\alpha} \overline{u},~~\mbox{for some} ~a \in A, \alpha \in \mathbb{Z},
$$
where
$$
u = a_1^{\alpha_1} a_2^{\alpha_2} \ldots a_k^{\alpha_k},~~a_i \in A, \alpha_i
\in \mathbb{Z}
$$
is a word and
$$
\overline{u} = a_k^{\alpha_k} a_{k-1}^{\alpha_{k-1}} \ldots a_1^{\alpha_1}
$$
is its reverse word.  Clearly,  ${\overline u}^{-1}= \overline {u^{-1} }$.
 
\medskip The following lemma is easy to prove. 
\begin{lemma}\label{onto}
Let $G = \langle A \rangle$ and $H = \langle B \rangle$ be two groups,
$\mathcal{P}(A)$ is the set of palindromes in the alphabet $A^{\pm 1},$
$\mathcal{P}(B)$ is the set of palindromes in the alphabet $B^{\pm 1}.$ If
$\varphi : G \longrightarrow H$ is an epimorphism such that
$\varphi(A) = B$ then
$$
{\rm pw}(H) \leq {\rm pw}(G).
$$
\end{lemma}
For free nilpotent groups of rank $n$ we have the following set of epimorphisms
$$
{\rm N}_{n,1} \longleftarrow {\rm N}_{n,2} \longleftarrow {\rm N}_{n,3} \longleftarrow \ldots
$$
where
$$
{\rm N}_{n,1} = {\rm N}_{n,2} / \gamma_2({\rm N}_{n,2}),~~~{\rm N}_{n,2} = {\rm N}_{n,3} / \gamma_3({\rm N}_{n,3}), \ldots.
$$
Applying the above lemma we have:
\begin{cor}\label{c1}  The following
inequalities hold:

\begin{equation*}\label{e2}
{\rm pw}({\rm N}_{n,1}) \leq {\rm pw}({\rm N}_{n,2}) \leq {\rm pw}({\rm N}_{n,3}) \leq \ldots
\end{equation*}
\end{cor}

\begin{lemma} \label{l:32}
Let $G=\langle A \rangle$ be a group generated by a set $A$. Then the following
hold.
\begin{enumerate}
\item{ If $p$ is a palindrome, then for $m$ in $\mathbb Z$, $p^m$ is also a palindrome.}

\item{ Any element in $G$ which is conjugate to a product of $n$ palindromes, $n
\geq 1$,  is a product of $n$ palindromes if $n$ is even, and of $n+1$ palindromes
if $n$ is odd. }

\item{Any commutator of the type  $[u, p],$  where $p$ is a palindrome is
a product of $3$ palindromes.
Any element $[u, a^{\alpha}] a^{\beta}$, $a \in A$, $\alpha, \beta \in
\mathbb{Z},$ is a product of $3$ palindromes.}

\item{In $G$ any commutator of the type  $[u, p q],$  where $p, q$ are
palindromes is a product of $4$ palindromes.
Any element $[u, p a^{\alpha}] a^{\beta}$, $a \in A$, $\alpha, \beta \in
\mathbb{Z},$ is a product of $4$ palindromes.}
\end{enumerate}

\end{lemma}

\begin{proof}

1) Let $p=u a^{\alpha} \overline u$, where $u$ is as above.
Then $p^2 = u a^{\alpha} \overline u u a^{\alpha} \overline u$, $p^3= u
a^{\alpha} \overline u u a^{\alpha} \overline u u a^{\alpha} \overline u$ are
palindromes.  The result now follows by induction.

\medskip 2) Let $v = u^{-1} p u$ be conjugate to a palindrome $p$. If $\overline{u}$
 is the reverse word of $u$, then
$$
v = (u^{-1} \, p \, \overline{u^{-1}}) \cdot \overline{u} \, u
$$
and we see that $u^{-1} \, p \, \overline{u^{-1}}$ and $\overline{u} \, u$ are
palindromes.

If $v$ is the conjugate to the product of $2m$ palindromes $p_1, \cdots,
p_{2m}$, for some $u \in G$ let
$v=u^{-1} p_1 p_2 \ldots p_{2m} u$. Then
$$
v = (u^{-1} \, p_1 \, \overline{u^{-1}}) (\overline u p_2 u) (u^{-1} p_3
\overline{u^{-1}})  \ldots  \overline{u} \, p_{2m} \, u
$$
is the product of $2m$ palindromes.

If $v=u^{-1} p_1 p_2 \ldots p_{2m} p_{2m+1} u$, then
$$
v=u^{-1} p_1 \ldots p_{2m}
\overline{u^{-1}} \overline u p_{2m+1} u = u^{-1} p_1 \ldots p_{2m}
u \cdot (u^{-1} \overline{u^{-1}}) \cdot \overline u p_{2m+1} u.
$$
By part (2) above,  $u^{-1} p_1 \ldots p_{2m}
u$ is a product of $2m$ palindromes, $u^{-1} \overline{u^{-1}}$ and $\overline u p_{2m+1} u$ are palindromes.
Hence, $v$ is a product of $2m+2$ palindromes.

\medskip 3) We can check that
$$
[u, p] = u^{-1} \, p^{-1} \, u \, p = u^{-1} \, p^{-1} \, \overline{u^{-1}}
\cdot \overline{u} \, u \cdot p
$$
and $u^{-1} \, p^{-1} \, \overline{u^{-1}}$ and $\overline{u} \, u $ are
palindromes. If we take $p = a^{\alpha}$ then it is clear that
$[u, a^{\alpha}] a^{\beta}$ is the product of three palindromes.

\medskip 4) Observe that $[u, pq]=(u^{-1} (q^{-1} p^{-1})u)pq$. By (2), $u^{-1} (pq)^{-1} u$ is
a product of two palindromes,  hence the result follows.
\end{proof}

\begin{prop}
Let $G$ be a group and let an element $g$ in the center of $G$ be a product of 2 palindromes. Then for any integer
$m$ the power $g^m$ is a product of 2 palindromes.
\end{prop}

\begin{proof}
Let $m > 0.$ Use induction on $m$. If $g = p_1 p_2$ is a product of 2 palindromes then
$$
g^2 = p_1 p_2 \cdot g = p_1 g p_2 = p_1^2 \cdot p_2^2
$$
and by Lemma \ref{l:32}(1) $p_1^2$ and $p_2^2$ are palindromes. Assume the result for some $m$. Then
$$
g^{m+1} = (p_1 p_2)^m \cdot g = p_1^m g p_2^m = p_1^{m+1} \cdot p_2^{m+1}
$$
and by Lemma \ref{l:32}(1) $p_1^{m+1}$ and $p_2^{m+1}$ are palindromes.

If $m < 0$ then $g^{m} = (g^{-m})^{-1}$ and the result follows from the previous case and the fact that the  inverse of a  palindrome is a palindrome.
\end{proof}

\section{Proof of \thmref{mainth}}\label{fn1}

\subsection{The palindromic widths of free nilpotent groups}
Let ${\rm N}_{n,r}$ be a free
$r$-step nilpotent group of rank $n \geq 2$.  Let $x_1, x_2, \ldots, x_n$ be a basis for ${\rm N}_{n, r}$. Let $\mathcal{P}$ be the set of all
palindromes in ${\rm N}_{n,r}$. Note that an element $p \in {\rm N}_{n,r}$
is a palindrome if it can be be represented in the form
\begin{equation}\label{palin0}
p = x_{i_1}^{\alpha_1} x_{i_2}^{\alpha_2} \ldots x_{i_k}^{\alpha_k}
x_{i_{k+1}}^{\alpha_{k+1}} x_{i_k}^{\alpha_k} \ldots
x_{i_2}^{\alpha_2} x_{i_1}^{\alpha_1}
\end{equation}
where
$$
~~i_j \in \{1, 2, \ldots, n \},~~\alpha_j \in \mathbb{Z} \setminus \{ 0 \}.
$$
\begin{lemma}\label{n1} ${\rm pw}({\rm N}_{n,1})=n$.\end{lemma}
\begin{proof}  In this case ${\rm N}_{n,1}$ is a free abelian group of rank $n$. Since any element $g \in {\rm N}_{n,1}$ has the form
\begin{equation}\label{palin1}
g = x_1^{\alpha_1} x_2^{\alpha_2} \ldots x_n^{\alpha_n},~~~\alpha_i \in
\mathbb{Z},
\end{equation}
then $g$ is a product of $n$ palindromes $x_i^{\alpha_i}$ and hence
$$
{\rm pw}({\rm N}_{n,1}) \leq n.
$$

To prove the equality we shall show that $l_{\mathcal{P}}(x_1 x_2 \ldots x_n) = n$. For this  define a map
$$
{\widehat{ }}: {\rm N}_{n,1} \longrightarrow \underbrace{\mathbb{Z}_2 \times \mathbb{Z}_2 \times
\ldots \times \mathbb{Z}_2}_n,
$$
where $\mathbb{Z}_2$ is a cyclic group of order $2$ by the rule
$$
\widehat{g} = (\varepsilon_1, \varepsilon_2, \ldots, \varepsilon_n),
$$
where
$$
\varepsilon_i =
\left\{\begin{array}{ll}
0 & ~\mbox{if}~\alpha_i~\mbox{ in \eqnref{palin1} is even}, \\
1 & ~\mbox{if}~\alpha_i~\mbox{in \eqnref{palin1} is odd}. \\
\end{array}
\right.
$$
Evidently, that for any $g, h \in G$ we have $\widehat{g h} = \widehat{g} +
\widehat{h}.$

If a palindrome $p$ has the form \eqnref{palin1} then
$$
\widehat{p} = (\nu_1, \nu_2, \ldots, \nu_n)
$$
contains no more than one non-zero component. On the other hand, for $w=x_1 x_2 \ldots x_n$, 
$$
\widehat{w}= (1, 1, \ldots, 1).
$$
Thus $x_1 x_2 \ldots x_n$ is a product of at least $n$ palindromes.
\end{proof}
\begin{lemma}\label{n2}
For $r \geq 2$, $n \leq {\rm pw}({\rm N}_{n,r}) \leq 3n$.\end{lemma}
\begin{proof}We claim that any element $g$ in ${\rm N}_{n,r}$ can be represented in the form
$$g=[u_1, x_1]x_1^{\alpha_1} [u_2, x_2] x_2^{\alpha_2} \ldots [u_n, x_n] \, x_n^{\alpha_n}.$$
Indeed, use induction on the step of
nilpotency $r.$ If $r=2$ and $g \in {\rm N}_{n,2}$ then by  \hbox{\lemref{ar1}},
$$
g = x_1^{\alpha_1}\, x_2^{\alpha_2}\, \ldots \, x_n^{\alpha_n} \, [u_1, x_1] \, [u_2, x_2] \, \ldots \,
[u_n, x_n],~~~\alpha_i \in \mathbb{Z},~~u_i \in {\rm N}_{n,2}.
$$
But the commutators $[u_i, x_i],$ $i =1, 2, \ldots, n,$ lie in the center of ${\rm N}_{n,2}$. Hence
$$
g = [u_1, x_1] \, x_1^{\alpha_1}\, [u_2, x_2]\, x_2^{\alpha_2}\,  \ldots \,
[u_n, x_n] \, x_n^{\alpha_n}.
$$
has the required form. Let the result holds for groups  ${\rm N}_{n,r}$.  We claim that the result also holds for ${\rm N}_{n,r+1}.$ Let
$\Gamma = \gamma_{r+1}({\rm N}_{n,r+1}) = [\gamma_r({\rm N}_{n,r+1}), {\rm N}_{n,r+1}].$ Then an element $g$ of ${\rm N}_{n,r+1}$
has the form
$$
g = [u_1, x_1] \, x_1^{\alpha_1}\, [u_2, x_2]\, x_2^{\alpha_2}\,  \ldots \,
[u_n, x_n] \, x_n^{\alpha_n} \, d
$$
for some $d \in \Gamma.$ By \lemref{reh},
$$
d = [a_1, x_1] \, [a_2, x_2] \, \ldots \, [a_n, x_n],~~\mbox{for some}~ a_i \in \gamma_r({\rm N}_{n,r+1}).
$$
Since all $[a_i, x_i]$ lie in the center of ${\rm N}_{n,r+1}$ then
$$
g =
[u_1, x_1] \, [a_1, x_1] \, x_1^{\alpha_1}\, [u_2, x_2]\, [a_2, x_2] \, x_2^{\alpha_2}\,  \ldots \,
[u_n, x_n] \, [a_1, x_1]  \, x_n^{\alpha_n} =
$$
$$
= [u_1 a_1, x_1] \,  x_1^{\alpha_1}\, [u_2 a_2, x_2]\,  x_2^{\alpha_2}\,  \ldots \,
[u_n a_n, x_n]   \, x_n^{\alpha_n}
$$
has the required form.

By \lemref{l:32}(3) any element $[u_i, x_i] \, x_i^{\alpha_i}$ is a product of 3 palindromes and $g$ is a
product of $3n$ palindromes. The lower bound follows from \lemref{n1} and \corref{c1}.
\end{proof}

\subsection{The palindromic widths of $2$-step free nilpotent groups}\label{fn2}
In what follows we will consider  2-step nilpotent group
${\rm N}_{n,2}$.
We know that any palindrome has the form
$p = u x_{l}^{\beta} \overline{u}$, where
$$
u = x_{i_1}^{\alpha_1} x_{i_2}^{\alpha_2} \ldots x_{i_k}^{\alpha_k}
$$
and
$$
\overline{u} = x_{i_k}^{\alpha_k} x_{i_{k-1}}^{\alpha_{k-1}} \ldots
x_{i_1}^{\alpha_1}
$$
is its reverse. We can assume that
$$
i_1 < i_2 < \ldots < i_k
$$
and
$$
l \not\in \{ i_1,  i_2, \ldots, i_k \}.
$$
Indeed, if
$$
p = u_1 \, x_i^{\alpha_i} \, x_j^{\alpha_j} \, p_0 \, x_j^{\alpha_j} \,
x_i^{\alpha_i} \, \overline{u_1}
$$
for some palindrome $p_0$ then
$$
p = u_1 \, x_j^{\alpha_j} \, x_i^{\alpha_i} \, [x_i, x_j]^{\alpha_i \alpha_j} \,
p_0 \, x_i^{\alpha_i} \, x_j^{\alpha_j} \,
[x_j, x_i]^{\alpha_i \alpha_j} \, \overline{u_1} =
u_1 \, x_j^{\alpha_j} \, x_i^{\alpha_i} \, p_0 \, x_i^{\alpha_i} \,
x_j^{\alpha_j} \, \overline{u_1}.
$$
Hence, we can permute any two letters in a word representing $u$ and the element $p$ does not change.

Let ${\rm N}_{2,2} = \langle x, y \rangle$ be the free nilpotent group of rank 2. Any
element in this group has the form
$$
x^{\alpha} y^{\beta} [y, x]^{\gamma},~~~\alpha, \beta, \gamma \in \mathbb{Z}.
$$

\begin{lemma}
For some integers $a$ and $b$, any palindrome in ${\rm N}_{2,2}$ has one of the following forms:
$$
p_{(2a, b)} = x^{2a} \, y^b \, z^{ab},~~~p_{(a, 2b)}= x^{a} y^{2b} z^{ab}, ~\mbox{where}~z =
[y, x].
$$\end{lemma}

\begin{proof}
Let $p$ be a palindrome in ${\rm N}_{2,2}$. Induction on the syllable length of
$p$. If it is equal to 1 then $p = x^a$ or $p = y^b.$
If the syllable length is 3 then,
$$
p = x^{\alpha} y^{\beta}  x^{\alpha} = x^{2\alpha}  y^{\beta}  [y, x]^{\alpha
\beta}
$$
or
$$
p = y^{\alpha} x^{\beta}  y^{\alpha} = x^{\beta}  y^{2 \alpha}  [y, x]^{\alpha
\beta}.
$$

Using the note before the lemma, we see that all other possibilities are reduced
to these two cases.
\end{proof}

We see that if a palindrome lies in the commutator subgroup ${\rm N}_{2,2}'$ then the palindrome is
trivial. More generally, we have the following.

\begin{lemma}\label{l:1}
If a product of two palindromes lies in ${\rm N}_{2,2}'$ then this product is trivial.
\end{lemma}

\begin{proof}
We know that any palindrome has the form $p_{(2a, b)}$ or $p_{(a, 2b)}$. Consider the product of
two palindromes.  If both palindromes have the type the $p_{(2a, b)}$
then their product
$$
p_{(2a_1, b_1)}\cdot p_{(2a_2, b_2)} = x^{2a_1} y^{b_1} z^{a_1 b_1} \cdot x^{2a_2} y^{b_2} z^{a_2 b_2}
= x^{2(a_1+a_2)} y^{b_1+b_2} z^{b_1(a_1+2a_2)+a_2b_2}
$$
lies in the commutator subgroup if and only if
$$
\left\{
  \begin{array}{c}
a_1 + a_2 = 0, \\
b_1 + b_2 = 0, \\
  \end{array}
\right.
$$
or
$$
\left\{
  \begin{array}{c}
a_1 = -a_2, \\
b_1 = -b_2. \\
  \end{array}
\right.
$$
But this means that
$$
p_{(2a_1, b_1)}\cdot p_{(2a_2, b_2)}= z^{-b_2 a_2 + a_2 b_2} = z^0 = 1.
$$
The case of a product $p_{(a_1, 2b_1)} \cdot p_{(a_2, 2b_2)}$ is similar.

Consider a product of palindromes of different types:
$$
p_{(2a_1, b_1)}\cdot p_{(a_2, 2b_2)}= x^{2a_1} y^{b_1} z^{a_1 b_1} \cdot x^{a_2} y^{2b_2} z^{a_2 b_2}
= x^{2a_1+a_2} y^{b_1+2b_2} z^{a_1b_1+a_2b_2+b_1a_2}.
$$
We see that this product lies in the commutator subgroup if and only if
$$
\left\{
  \begin{array}{c}
2a_1 + a_2 = 0, \\
b_1 + 2b_2 = 0, \\
  \end{array}
\right.
$$
or
$$
\left\{
  \begin{array}{c}
a_2 = -2a_1, \\
b_1 = -2b_2. \\
  \end{array}
\right.
$$
But this means that
$$
p_{(2a_1, b_1) } \cdot p_{(a_2, 2b_2)} = z^{b_1(a_1+a_2) + a_2 b_2} = z^0 = 1.
$$
The case of the  product $p_{(a_2, 2b_2)} \cdot p_{(2a_1, b_1)}$ is similar.
\end{proof}

\begin{prop}\label{t:21}
${\rm pw}({\rm N}_{2,2}) = 3$.
\end{prop}

\begin{proof}
 Note that $[y, x]$ is an element in the center of ${\rm N}_{2,2}$. Note that
\begin{eqnarray*}
 x^{\alpha} y^{\beta} [y, x]^{\gamma} &=& x^{\alpha} y^{\beta-\gamma} y^{\gamma}
[y, x]^{\gamma}\\
 &=& x^{\alpha} y^{\beta-\gamma} (x^{-1} y x)^{\gamma}. \\
\end{eqnarray*}
It follows that
$$
x^{\alpha} y^{\beta} z^{\gamma} = x^{\alpha}  y^{\beta-\gamma}
y^{\gamma}x^{\alpha} \cdot x^{-\alpha-2} \cdot x y^{\gamma} x.
$$
Hence, ${\rm pw}({\rm N}_{2,2}) \leq 3.$ On the other hand, it follows from Lemma \ref{l:1} that
$z$ is not a product of $2$ palindromes. Hence ${\rm pw}({\rm N}_{2,2}) \geq 3.$
\end{proof}

In the general case we can prove

\begin{lemma}\label{n3}
Any element in ${\rm N}_{n,2},$ $n \geq 2$ is a product of at most $3(n-1)$ palindromes.
\end{lemma}

\begin{proof}
Let $g \in {\rm N}_{n,2}.$ Then $g$ has the form
$$
g = x_1^{\alpha_1} x_2^{\alpha_2} \ldots  x_n^{\alpha_n}
\prod_{1\leq j < i\leq n} [x_i, x_j]^{\gamma_{ij}}
$$
for some integers $\alpha_i$ and $\gamma_{ij}.$ Using the commutator  identities (see   e.g. 
 \cite{MKS}) we have
$$
\prod_{1\leq j < i\leq n} [x_i, x_j]^{\gamma_{ij}} =
[x_n^{\gamma_{n1}} x_{n-1}^{\gamma_{n-1,1}} \ldots x_2^{\gamma_{21}}, x_1]
[x_n^{\gamma_{n2}} x_{n-1}^{\gamma_{n-1,2}} \ldots x_3^{\gamma_{32}}, x_2] \ldots
$$
$$
[x_n^{\gamma_{n,n-2}} x_{n-1}^{\gamma_{n-1,n-2}}, x_{n-2}]
[x_n^{\gamma_{n,n-1}}, x_{n-1}]
$$
Since, the commutator subgroup ${\rm N}_{n,2}'$ is equal to the center of ${\rm N}_{n,2}$ then
$$
g =
[x_n^{\gamma_{n1}} x_{n-1}^{\gamma_{n-1,1}} \ldots x_2^{\gamma_{21}}, x_1]x_1^{\alpha_1} \cdot
[x_n^{\gamma_{n2}} x_{n-1}^{\gamma_{n-1,2}} \ldots x_3^{\gamma_{32}}, x_2] x_2^{\alpha_2} \cdot \ldots \cdot
$$
$$
[x_n^{\gamma_{n,n-2}} x_{n-1}^{\gamma_{n-1,n-2}}, x_{n-2}] x_{n-2}^{\alpha_{n-2}} \cdot x_{n-1}^{\alpha_{n-1}}
[x_n^{\gamma_{n,n-1}}, x_{n-1}] x_n^{\alpha_n}.
$$
By Lemma \ref{l:32}(2) any element
$$
[x_n^{\gamma_{n1}} x_{n-1}^{\gamma_{n-1,1}} \ldots x_2^{\gamma_{21}}, x_1]x_1^{\alpha_1},~~~
[x_n^{\gamma_{n2}} x_{n-1}^{\gamma_{n-1,2}} \ldots x_3^{\gamma_{32}}, x_2] x_2^{\alpha_2},~~\ldots,
[x_n^{\gamma_{n,n-2}} x_{n-1}^{\gamma_{n-1,n-2}}, x_{n-2}] x_{n-2}^{\alpha_{n-2}}
$$
is a product of 3 palindromes. Elements $x_{n-1}$ and $x_n$ generate a group which
is isomorphic to ${\rm N}_{2,2}$ and by  \propref{t:21}, the element $x_{n-1}^{\alpha_{n-1}}
[x_n^{\gamma_{n,n-1}}, x_{n-1}] x_n^{\alpha_n}$ is a product of 3 palindromes. Hence, $g$ is a product
of $3(n-1)$ palindromes.
\end{proof}

 \subsection{Proof of \thmref{mainth}}
\thmref{mainth} is obtained by combining  \lemref{n1}, \lemref{n2} and \lemref{n3}.

 \subsection{The palindromic widths of free abelian-by-nilpotent groups}
\begin{prop}\label{fabn}

Let $G$ be a non-abelian free abelian-by-nilpotent group of rank $n$.  Then  ${\rm pw}(G) \leq 5n$.
\end{prop}

\begin{proof}
Let  $G$ be a non-abelian free abelian-by-nilpotent group with a basis $x_1, x_2, \ldots, x_n$.
Let $A$ be an abelian normal subgroup of $G$ such that $G / A$ is nilpotent.
It follows from \cite[Theorem 2]{MR} that any element $g \in G$ has the form
$$
g = x_1^{\alpha_1}\, x_2^{\alpha_2}\, \ldots \, x_n^{\alpha_n} \, [u_1, x_1]^{a_1} \, [u_2, x_2]^{a_2} \, \ldots \,
[u_n, x_n]^{a_n},~~~\alpha_i \in \mathbb{Z},~~u_i \in G,~~a_i \in A.
$$
By (3) of \lemref{l:32} any commutator $[u_i, x_i]$ is a product of 3 palindromes,  thus
by (2) of Lemma \ref{l:32}, any commutator $[u_i, x_i]^{a_i}$ is a product of 4 palindromes.
Hence, $g$ is a product of $n + 4n = 5n$ palindromes.
\end{proof}
Since free metabelian groups are free abelian-by-abelian groups, hence we have the following. 
\begin{cor}
The palindromic width of a finitely generated free metabelian group is finite. 
\end{cor}

\end{document}